\documentclass[11pt]{amsart}
\usepackage{amsfonts}
\usepackage{amssymb}
\usepackage{graphicx}%
\usepackage{tikz}
\usetikzlibrary{shapes.geometric, arrows}
\tikzstyle{io} = [rectangle, rounded corners, minimum width=2cm, minimum height=1cm,text centered, draw=black, fill=white]
\tikzstyle{arrow} = [thick,->,>=stealth]

\usepackage{mathtools}
\usepackage{color} 
\usepackage[utf8]{inputenc} 

\usepackage{amscd}
\usepackage{mathrsfs}
\newtheorem{theorem}{Theorem}[section]
\newtheorem{lemma}[theorem]{Lemma}

\numberwithin{equation}{section}

\renewenvironment{proof}[1][\proofname]{%
  \par\pushQED{\qed}\normalfont%
  \trivlist\item[\hskip\labelsep\bfseries#1{.}]%
}{%
  \popQED
}

\begin{document}
\title[Symplectic Balls]{On Orthogonal Projections of Symplectic balls}
\author{Nuno C. Dias}

\author{Maurice A. de Gosson}

\author{Jo\~{a}o N. Prata}

\date{}

\begin{abstract}
We study the orthogonal projections of symplectic balls in $\mathbb{R}^{2n}$
on complex subspaces. In particular we show that these projections are
themselves symplectic balls under a certain complexity assumption. Our main
result is a refinement of a recent very interesting result of Abbondandolo and
Matveyev extending the linear version of Gromov's non-squeezing theorem. We
use a conceptually simpler approach where the Schur complement of a matrix
plays a central role.

\end{abstract}
\maketitle
\tableofcontents

\section{Introduction}

\subsection{What is known}

Let $\sigma=dp_{1}\wedge dx_{1}+\cdot\cdot\cdot+dp_{n}\wedge dx_{n}$ be the
standard symplectic form on $\mathbb{R}^{2n}\equiv\mathbb{R}^{n}%
\times\mathbb{R}^{n}$; we call \emph{symplectic ball} the image of the ball
\[
B^{2n}(z_{0},R)=\{z\in\mathbb{R}^{2n}:|z-z_{0}|\leq R\}
\]
by a symplectic automorphism $S\in\operatorname*{Sp}(n)$ (the symplectic group
of $(\mathbb{R}^{2n},\sigma)$). As a consequence of Gromov's non-squeezing
theorem \cite{Gromov,HZ} the orthogonal projection of a symplectic ball
$S(B^{2n}(z_{0},R))$ on any two-dimensional symplectic subspace of
$(\mathbb{R}^{2n},\sigma)$ has area at least equal to $\pi R^{2}$.
Abbondandolo and Matveyev asked in \cite{abbo} the question whether this
result could be generalized to subspaces with higher dimensions. They showed
that the orthogonal projection $\Pi_{\mathbb{V}}S(B^{2n}(z_{0},R))$ of
$S(B^{2n}(z_{0},R))$ onto an arbitrary complex symplectic subspace
$(\mathbb{V},\sigma_{|\mathbb{V}})$ of $(\mathbb{R}^{2n},\sigma)$ such that
$\dim\mathbb{V}=2k$ satisfies
\begin{equation}
\operatorname*{Vol}\nolimits_{\mathbb{V}}\Pi_{\mathbb{V}}S(B^{2n}%
(z_{0},R))\geq\frac{(\pi R^{2})^{k}}{k!}\label{abbo}%
\end{equation}
where $\operatorname*{Vol}\nolimits_{\mathbb{V}}$ is the volume element on
$\mathbb{V}$. Notice that $(\pi R^{2})^{k}/k!$ is the volume of the ball
$B_{\mathbb{V}}(\Pi_{\mathbb{V}}z_{0},R)$ in $\mathbb{V}$:
\begin{equation}
\operatorname*{Vol}\nolimits_{\mathbb{V}}(B_{\mathbb{V}}(\Pi_{\mathbb{V}}%
z_{0},R))=\frac{(\pi R^{2})^{k}}{k!}.\label{abbobis}%
\end{equation}
They moreover proved that equality holds in (\ref{abbo}) if and only
$S^{T}\mathbb{V}$ is itself a complex subspace of $\mathbb{R}^{2n}$. The
inequality (\ref{abbo}) implies the linear version of Gromov's theorem when
$\dim\mathbb{V=}2$ and conservation of volume by linear symplectomorphisms
when $\mathbb{V}=\mathbb{R}^{2n}$. Abbondandolo and and Matveyev proved their
results using an ingenuous but complicated argument involving the Wirtinger
inequality for 2-forms on K\"{a}hler manifolds \cite{Fed69}. Results of this
type are more subtle and difficult than they might appear at first sight; for
instance as Abbondandolo and Matveyev show the inequality (\ref{abbo}) does
not hold when one replaces $S$ by a nonlinear symplectomorphism $f$. In fact,
one can construct examples where $\operatorname*{Vol}\nolimits_{\mathbb{V}}%
\Pi_{\mathbb{V}}f(B^{2n}(R))$ can become arbitrarily small. They however make
an interesting conjecture, to which we will come back at the end of this paper.

\subsection{What we will do}

We will prove by elementary means a stronger version of (\ref{abbo}) and of
its extension. We will actually prove (Theorem \ref{Prop1}) that the
orthogonal projection of a symplectic ball on a symplectic subspace contains a
symplectic ball with the same radius in this subspace, and is itself a
symplectic ball when the subspace under consideration is complex. The proof
will be done in the particular case where the symplectic space $\mathbb{V}$ is
of the type $\mathbb{R}^{2n_{A}}\oplus0$ in which case the symplectic
orthogonal $\mathbb{V}^{\sigma}$ is $0\oplus\mathbb{R}^{2n_{B}}$; our
refinement of (\ref{abbo}) says that for every $S\in\operatorname*{Sp}(n)$
there exists $S_{A}\in\operatorname*{Sp}(n_{A})$ (the symplectic group of
$\mathbb{R}^{2n_{A}}\oplus0\equiv\mathbb{R}^{2n_{A}}$\footnote{For the sake of
simplicity, we make the identification $\mathbb{R}^{2n_{A}}\oplus
0\equiv\mathbb{R}^{2n_{A}}$. In particular, we write by abuse of language
$\Pi_{\mathbb{V}}(z_{A},z_{B})=z_{A}$ instead of $(z_{A},0)$.}) such that
\begin{equation}
\Pi_{\mathbb{V}}S(B^{2n}(z_{0},R))\supseteq S_{A}(B^{2n_{A}}(z_{0,A},R))
\label{abboprime}%
\end{equation}
and $z_{0,A}=\Pi_{\mathbb{V}}z_{0}$.

This will be done using the theory of Schur complements and the notion of
symplectic spectrum of a positive definite matrix. Since symplectomorphisms
are volume-preserving, (\ref{abboprime}) implies (\ref{abbo}). It is however a
much stronger statement than (\ref{abbo}) because, given two measurable sets
$\Omega$ and $\Omega^{\prime}$ with the same volume, there does not in general
exist a symplectomorphism (let alone a linear one) taking $\Omega$ to
$\Omega^{\prime}$ as soon as the dimension of the symplectic space exceeds two
\cite{HZ}.

Note that these results are invariant under phase space translations. We will
therefore assume henceforth that $z_{0}=0$.

We finally discuss in section \ref{SecPerspectives} some possible extensions
to the non-linear case, pointing out the difficulties.

\section{Preliminaries\label{secone}}

In what follows $M$ will be a real $2n\times2n$ positive definite matrix; we
will write $M>0$. We denote by $J$ the standard symplectic matrix $%
\begin{pmatrix}
0_{n} & I_{n}\\
-I_{n} & 0_{n}%
\end{pmatrix}
$. We have, $\sigma(z,z^{\prime})=Jz\cdot z^{\prime}$ when $z=(x,p)$,
$z^{\prime}=(x^{\prime},p^{\prime})$. In this notation the condition
$S\in\operatorname*{Sp}(n)$ is equivalent to $S^{T}JS=J$ (or $SJS^{T}=J$)
where $S^{T}$ is the transpose of $S$.

\subsection{Williamson's symplectic diagonalization}

By definition the symplectic spectrum of $M$ is the increasing sequence
$\lambda_{1}^{\sigma}(M)\leq$ $\lambda_{2}^{\sigma}(M)\leq\cdot\cdot\cdot
\leq\lambda_{n}^{\sigma}(M)$ of numbers $\lambda_{j}^{\sigma}(M)>0$ where the
$\pm i\lambda_{j}^{\sigma}(M)$ are the eigenvalues of $JM$ (which are the same
as those of the antisymmetric matrix $M^{1/2}JM^{1/2}$). We will use the
following property, known in the literature as \textquotedblleft Williamson's
symplectic diagonalization theorem\textquotedblright\ \cite{Birk,HZ}: there
exists $S\in\operatorname*{Sp}(n)$ such that
\[
M=S^{T}DS\text{ , }D=%
\begin{pmatrix}
\Lambda & 0\\
0 & \Lambda
\end{pmatrix}
\]
where $\Lambda$ is the diagonal matrix whose eigenvalues are the numbers
$\lambda_{j}^{\sigma}(M)$ (all matrices corresponding here to the standard
splitting $z=(x,p)$). The symplectic spectra of $M$ and $M^{-1}$ are inverses
of each other in the sense that:
\begin{equation}
\lambda_{j}^{\sigma}(M^{-1})=\lambda_{n-j}^{\sigma}(M)^{-1}\text{
\ \textit{for }\ }1\leq j\leq n. \label{invspec}%
\end{equation}
We also have the less obvious property (\cite{Birk},  section 8.3.2)
\begin{equation}
M\leq N\Longrightarrow\lambda_{j}^{\sigma}(M)\leq\lambda_{j}^{\sigma}(N)\text{
\ \textit{for }\ }1\leq j\leq n \label{MN}%
\end{equation}
where $M\leq N$ means $N-M\geq0$.

The following simple result characterizing positive semi-definiteness in terms
of the symplectic spectrum will be very useful for proving Theorem \ref{Prop1}:

\begin{lemma}
\label{Lemma}The Hermitian matrix $M+iJ$ is positive semi-definite:
$M+iJ\geq0$ if and only if $\lambda_{j}^{\sigma}(M)\geq1$ for $1\leq j\leq n$.
\end{lemma}

\begin{proof}
Let $M=S^{T}DS$ be a Williamson diagonalization of $M$; since $S^{T}JS=J$ the
condition $M+iJ\geq0$ is equivalent to $D+iJ\geq0$. The characteristic
polynomial of $D+iJ$ is the product $P(\lambda)=P_{1}(\lambda)\cdot\cdot\cdot
P_{n}(\lambda)$ of polynomials%
\[
P_{j}(\lambda)=\lambda^{2}-2\lambda_{j}^{\sigma}(M)\lambda+\lambda_{j}%
^{\sigma}(M)^{2}-1
\]
and the eigenvalues of $M+iJ$ are thus the numbers $\lambda_{j}=\lambda
_{j}^{\sigma}(M)\pm1$. The condition $M+iJ\geq0$ is equivalent to $\lambda
_{j}\geq0$, that is to $\lambda_{j}^{\sigma}(M)\geq1$ for $j=1,...,n$.
\end{proof}

Notice that if $S\in\operatorname*{Sp}(n)$ we have
\begin{equation}
S^{T}S+iJ\geq0\text{ \ and }SS^{T}+iJ\geq0\label{sij}%
\end{equation}
since $\lambda_{j}^{\sigma}(S^{T}S)=\lambda_{j}^{\sigma}(SS^{T})=1$ for all
$j$ (because $D=I$ in view of Williamson's diagonalization result).

\subsection{Block-matrix partitions and Schur complements}

Let $\mathbb{R}^{2n_{A}}\equiv\mathbb{R}^{2n_{A}}\oplus0$ and $\mathbb{R}%
^{2n_{B}}\equiv0\oplus\mathbb{R}^{2n_{B}}$ be two symplectic subspaces of
$\mathbb{R}^{2n}$. We split $(\mathbb{R}^{2n},\sigma)$ as a direct sum
$(\mathbb{R}^{2n_{A}}\oplus\mathbb{R}^{2n_{B}},\sigma_{A}\oplus\sigma_{B})$
where $\sigma_{A}$ and $\sigma_{B}$ are, respectively, the restrictions of
$\sigma$ to $\mathbb{R}^{2n_{A}}$ and $\mathbb{R}^{2n_{B}}$. We write
$z\in\mathbb{R}^{2n}$ as $z=(z_{A},z_{B})=z_{A}\oplus z_{B}$ with
$z_{A}=(x_{A},p_{A})\in\mathbb{R}^{2n_{A}}$ and $z_{B}=(x_{B},p_{B}%
)\in\mathbb{R}^{2n_{B}}$.

We denote by $\Pi_{A}$ (\textit{resp}. $\Pi_{B}$) the orthogonal projection
$\mathbb{R}^{2n}\longrightarrow\mathbb{R}^{2n_{A}}$ (\textit{resp.}
$\mathbb{R}^{2n_{B}}$). We choose symplectic bases $\mathcal{B}_{A}$,
$\mathcal{B}_{B}$ in $\mathbb{R}^{2n_{A}}$ and $\mathbb{R}^{2n_{B}}$ and
identify linear mappings $\mathbb{R}^{2n}\longrightarrow\mathbb{R}^{2n}$ with
their matrices in the symplectic basis $\mathcal{B}=\mathcal{B}_{A}%
\oplus\mathcal{B}_{B}$ of $\mathbb{R}^{2n}$. Such a matrix will be written as
\begin{equation}
M=%
\begin{pmatrix}
M_{AA} & M_{AB}\\
M_{BA} & M_{BB}%
\end{pmatrix}
\label{MC}%
\end{equation}
the blocks $M_{AA}$, $M_{AB}$, $M_{BA}$, $M_{BB}$ having dimensions
$2n_{A}\times2n_{A}$, $2n_{A}\times2n_{B}$, $2n_{B}\times2n_{A}$,
$2n_{B}\times2n_{B}$, respectively. Similarly, the standard symplectic matrix
$J$ will be split as
\[
J=J_{A}\oplus J_{B}\equiv%
\begin{pmatrix}
J_{A} & 0\\
0 & J_{B}%
\end{pmatrix}
\]
where $J_{A}$ (\textit{resp.} $J_{B}$) is the standard symplectic matrix in
$(\mathbb{R}^{2n_{A}},\sigma_{A})$ (\textit{resp}. $(\mathbb{R}^{2n_{B}%
},\sigma_{B})$).

Since $M$ is positive definite and symmetric the upper-left and lower-right
blocks in (\ref{MC}) are themselves positive-definite and symmetric:
$M_{AA}>0$ and $M_{BB}>0$. In particular the Schur complements%
\begin{align}
M/M_{BB}  &  =M_{AA}-M_{AB}M_{BB}^{-1}M_{BA}\label{SchurB}\\
M/M_{AA}  &  =M_{BB}-M_{BA}M_{AA}^{-1}M_{AB} \label{SchurA}%
\end{align}
are well defined and invertible \cite{Zhang}, and the inverse of the matrix
$M$ is given by the formula%
\begin{equation}
M^{-1}=%
\begin{pmatrix}
(M/M_{BB})^{-1} & -(M/M_{BB})^{-1}M_{AB}M_{BB}^{-1}\\
-M_{BB}^{-1}M_{BA}(M/M_{BB})^{-1} & (M/M_{AA})^{-1}%
\end{pmatrix}
. \label{Minv}%
\end{equation}

\subsection{Orthogonal projections of ellipsoids in $\mathbb{R}^{2n}$}

We will also need the following general characterization of the orthogonal
projection of an ellipsoid on a subspace:

\begin{lemma}
\label{Lemma1}Let $\Pi_{A}$ be the orthogonal projection $\mathbb{R}%
^{2n}\longrightarrow\mathbb{R}^{2n_{A}}$ and
\[
\Omega=\{z\in\mathbb{R}^{2n}:Mz^{2}\leq R^{2}\}.
\]
We have
\begin{equation}
\Pi_{A}\Omega=\{z_{A}\in\mathbb{R}^{2n_{A}}:(M/M_{BB})z_{A}^{2}\leq R^{2}\}.
\label{bounda}%
\end{equation}

\end{lemma}

\begin{proof}
Let us set $Q(z)=$ $Mz^{2}-R^{2}$; the boundary $\partial\Omega$ of the
hypersurface $Q(z)=0$ is defined by%
\begin{equation}
M_{AA}z_{A}^{2}+2M_{BA}z_{A}\cdot z_{B}+M_{BB}z_{B}^{2}=R^{2}. \label{maza}%
\end{equation}
A point $z_{A}$ belongs to $\partial\Pi_{A}\Omega$ if and only if the normal
vector to $\partial\Omega$ at the point $z=(z_{A},z_{B})$ is parallel to
$\mathbb{R}^{2n_{A}}$ hence the constraint
\[
\partial_{z}Q(z)=2Mz\in\mathbb{R}^{2n_{A}}\oplus0~.
\]
This is equivalent to the condition $M_{BA}z_{A}+M_{BB}z_{B}=0$, that is to
$z_{B}=-M_{BB}^{-1}M_{BA}z_{A}$. Inserting $z_{B}$ in (\ref{maza}) shows that
the boundary $\partial\Pi_{A}\Omega$ is the set of all $z_{A}\in
\mathbb{R}^{2n_{A}}$ such that $(M/M_{BB})z_{A}^{2}=R^{2}$ hence formula
(\ref{bounda}).
\end{proof}

Interchanging $A$ and $B$ the orthogonal projection of $\Omega$ on
$\mathbb{R}^{2n_{B}}$ is similarly given by%
\begin{equation}
\Pi_{B}\Omega=\{z_{B}\in\mathbb{R}^{2n_{B}}:(M/M_{AA})z_{B}^{2}\leq R^{2}\}.
\label{boundb}%
\end{equation}

\section{Orthogonal Projections of Symplectic Balls\label{sectwo}}

\subsection{The main result: statement and proof}

Let us now prove the main result. We assume again the matrix $M$ is written in
block-form (\ref{MC}). To simplify notation we also assume that all balls
$B^{2n}(z_{0},R)$ are centered at the origin and set $B^{2n}(0,R)=B^{2n}%
(R)$.The case of a general ball $B^{2n}(z_{0},R)$ trivially follows using the
translation $z\longmapsto z+z_{0}$.

\begin{theorem}
\label{Prop1}Let $S\in\operatorname*{Sp}(n)$. (i) There exists $S_{A}%
\in\operatorname*{Sp}(n_{A})$ such that
\begin{equation}
\Pi_{A}(S(B^{2n}(R))\supseteq S_{A}(B^{2n_{A}}(R)); \label{abbo1}%
\end{equation}
(ii) We have
\begin{equation}
\Pi_{A}(S(B^{2n}(R))=S_{A}(B^{2n_{A}}(R)) \label{abbo2}%
\end{equation}
if and only if $S=S_{A}\oplus S_{B}$ for some $S_{B}\in\operatorname*{Sp}%
(n_{B})$, in which case we also have%
\begin{equation}
\Pi_{B}(S(B^{2n}(R))=S_{B}(B^{2n_{B}}(R)) \label{abbo3}%
\end{equation}

\end{theorem}

\begin{proof}
\textit{(i)} The symplectic ball $S(B^{2n}(R))$ consists of all $z\in
\mathbb{R}^{2n}$ such that $Mz^{2}\leq R^{2}$ where $M=(SS^{T})^{-1}$. It
follows from Lemma \ref{Lemma1} that $\Pi_{A}S(B^{2n}(R))$ is determined by
the inequality $(M/M_{BB})z_{A}^{2}\leq R^{2}$. We are going to show that the
symplectic eigenvalues $\lambda_{j}^{\sigma_{A}}(M/M_{BB})$ are all $\leq1$.
The inclusion (\ref{abbo1}) will then follow since we have, in view of
Williamson's diagonalization result,
\begin{equation}
M/M_{BB}=(S_{A}^{-1})^{T}D_{A}S_{A}^{-1}\label{mmbsa}%
\end{equation}
for some $S_{A}\in\operatorname*{Sp}(n_{A})$ and
\begin{equation}
D_{A}=%
\begin{pmatrix}
\Lambda_{A} & 0\\
0 & \Lambda_{A}%
\end{pmatrix}
\text{ \ , \ }\Lambda_{A}=\operatorname*{diag}(\lambda_{1}^{\sigma_{A}%
}(M/M_{BB}),...,\lambda_{n_{A}}^{\sigma_{A}}(M/M_{BB})).\label{diaga}%
\end{equation}
It follows that:
\begin{align}
\left(  M/M_{BB}\right)  z_{A}^{2} &  =\left(  (S_{A}^{-1})^{T}D_{A}S_{A}%
^{-1}\right)  z_{A}^{2}\label{eqExtra1}\\
&  =D_{A}\left(  S_{A}^{-1}z_{A}\right)  ^{2}\leq|S_{A}^{-1}z_{A}%
|^{2}\nonumber
\end{align}
The condition $z\in S_{A}(B^{2n_{A}}(R))$ being equivalent to $|S_{A}%
^{-1}z_{A}|^{2}\leq R^{2}$ thus implies $(M/M_{BB})z_{A}^{2}\leq R^{2}$ and
hence $S_{A}(B^{2n_{A}}(R))\subseteq\Pi_{A}S(B^{2n}(R))$.

To prove that we indeed have%
\begin{equation}
\lambda_{j}^{\sigma_{A}}((M/M_{BB}))\leq1\text{ \ \textit{for} }1\leq j\leq
n_{A} \label{infone}%
\end{equation}
we begin by noting that the symplectic eigenvalues $\lambda_{j}^{\sigma}(M)$
of $M=(SS^{T})^{-1}$ are all trivially equal to one, and hence also those of
its inverse $M^{-1}$: $\lambda_{j}^{\sigma}(M^{-1})=1$ for $1\leq j\leq n$. In
view of Lemma \ref{Lemma} the Hermitian matrix $M^{-1}+iJ$ is positive
semidefinite:
\begin{equation}
M^{-1}+iJ\geq0 \label{MiJ}%
\end{equation}
which implies that (\textit{cf}.(\ref{Minv}))
\begin{equation}
(M/M_{BB})^{-1}+iJ_{A}\geq0 \label{mmj}%
\end{equation}
(recall that $J=J_{A}\oplus J_{B}$). Applying now Lemma \ref{Lemma} to
$(M/M_{BB})^{-1}$ this implies that the inequalities (\ref{infone}) must hold
(\textit{cf}.(\ref{invspec})).

\textit{(ii)} In view of (\ref{bounda}) in Lemma \ref{Lemma1} the equality
(\ref{abbo2}) will hold if and only if $M/M_{BB}=(S_{A}S_{A}^{T})^{-1}$. If
$S=S_{A} \oplus S_{B}$, then $M=(SS^{T})^{-1}$ implies $M/M_{BB}=(S_{A}%
S_{A}^{T})^{-1}$ and we have the equality (\ref{abbo2}).

Conversely, suppose the equality (\ref{abbo2}) holds. From the inversion
formula (\ref{Minv}) we have%
\begin{equation}
M^{-1}=%
\begin{pmatrix}
(M/M_{BB})^{-1} & X\\
X^{T} & (M/M_{AA})^{-1}%
\end{pmatrix}
.
\end{equation}
where $X=-(M/M_{BB})^{-1}M_{AB}M_{BB}^{-1}$; since $M^{-1}\in
\operatorname*{Sp}(n)$ is symmetric it satisfies the symplecticity condition
\[
M^{-1}(J_{A}\oplus J_{B})M^{-1}=J_{A}\oplus J_{B}%
\]
which is in turn equivalent to the set of conditions
\begin{gather}
(M/M_{BB})^{-1}J_{A}(M/M_{BB})^{-1}+XJ_{B}X^{T}=J_{A}\label{eqArray1}\\
(M/M_{BB})^{-1}J_{A}X+XJ_{B}(M/M_{AA})^{-1}=0\label{eqArray2}\\
X^{T}J_{A}X+(M/M_{AA})^{-1}J_{B}(M/M_{AA})^{-1}J_{B}.\label{eqArray3}%
\end{gather}
Since
\begin{equation}
M/M_{BB}=(S_{A}S_{A}^{T})^{-1}\in\operatorname*{Sp}(n_{A}),\label{construct}%
\end{equation}
it follows from (\ref{eqArray1}) that
\begin{equation}
XJ_{B}X^{T}=0.\label{eqArray4}%
\end{equation}
Multiplying the identity (\ref{eqArray2}) on the right by $(M/M_{AA})X^{T}$
and using (\ref{eqArray4}), we obtain%
\[
(M/M_{BB})^{-1}J_{A}X(M/M_{AA})X^{T}=0
\]
that is
\begin{equation}
X(M/M_{AA})X^{T}=0.\label{eqArray5}%
\end{equation}
Since $(M/M_{AA})>0$, this is possible if and only if $X=0$. Finally, from
(\ref{eqArray3}) we conclude that $(M/M_{AA})^{-1}\in\operatorname*{Sp}%
(n_{B})$. Moreover, since $(M/M_{AA})^{-1}$ is symmetric and positive
definite, there exists $S_{B}\in\operatorname*{Sp}(n_{B})$, such that
$(M/M_{AA})^{-1}=S_{B}S_{B}^{T}$.

Altogether
\[
M^{-1}=SS^{T}=%
\begin{pmatrix}
S_{A}S_{A}^{T} & 0\\
0 & S_{B}S_{B}^{T}%
\end{pmatrix}
\]
hence $S=S_{A}\oplus S_{B}$, which concludes the proof.

\end{proof}

We remark that the proof above actually provides the means to calculate
explicitly the symplectic automorphisms $S_{A}$ in (\ref{abbo1}).
Recapitulating, it is constructed as follows: given $S\in\operatorname*{Sp}%
(n)$ calculate
\[
M=(SS^{T})^{-1}=%
\begin{pmatrix}
M_{AA} & M_{AB}\\
M_{BA} & M_{BB}%
\end{pmatrix}
\]
and then obtain the Schur complement (\ref{SchurB})
\[
M/M_{BB}=M_{AA}-M_{AB}M_{BB}^{-1}M_{BA}.
\]
The matrix $S_{A}$ is then obtained from (\ref{construct}) (observe that
$S_{A}$ is only defined up to a symplectic rotation, but this ambiguity is
irrelevant since $B^{2n_{A}}(R)$ is rotationally invariant).

\subsection{Discussion and extension}

Theorem \ref{Prop1} implies \textit{de facto} the Abbondandolo and Matveyev
result (\ref{abbo}) since formula (\ref{abbo1}) has the immediate consequence
that
\begin{equation}
\operatorname*{Vol}\nolimits_{2n_{A}}\Pi_{A}(S(B^{2n}(R))\geq\frac{(\pi
R^{2})^{n_{A}}}{n_{A}!} \label{abboineq}%
\end{equation}
noting that $S_{A}\in\operatorname*{Sp}(n_{A})$ is volume-preserving.
Similarly, the equality (\ref{abbo2}) implies
\begin{equation}
\operatorname*{Vol}\nolimits_{2n_{A}}\Pi_{A}(S(B^{2n}(R))=\frac{(\pi
R^{2})^{n_{A}}}{n_{A}!}. \label{abboeq}%
\end{equation}

Abbondandolo and Matveyev's \cite{abbo} however prove these relations for
projections on a general complex symplectic subspace $\mathbb{V}$ of
$(\mathbb{R}^{2n},\sigma)$ (that is such that $J\mathbb{V}=\mathbb{V}$) and
they show that the equality in (\ref{abbo}) holds if and only if the subspace
$S^{T}\mathbb{V}$ is complex, that is if $JS^{T}\mathbb{V=}S^{T}\mathbb{V}$.
Let us check that these conditions are satisfied when $\mathbb{V}%
=\mathbb{R}^{2n_{A}}\oplus0$. We first note that%
\[
J\mathbb{V}=(J_{A}\oplus J_{B})(\mathbb{R}^{2n_{A}}\oplus0)=\mathbb{R}%
^{2n_{A}}\oplus0=\mathbb{V}%
\]
hence $\mathbb{R}^{2n_{A}}\oplus0$ is complex in the sense above. Similarly,
the condition $JS^{T}\mathbb{V}=S^{T}\mathbb{V}$ is satisfied if
$S=S_{A}\oplus S_{B}$ since
\[
(J_{A}\oplus J_{B})(S_{A}^{T}\oplus S_{B}^{T})(\mathbb{R}^{2n_{A}}%
\oplus0)=\mathbb{R}^{2n_{A}}\oplus0;
\]
conversely the condition $JS^{T}\mathbb{V}=S^{T}\mathbb{V}$ is equivalent to
\[
(SS^{T})^{-1}J\mathbb{V}=\mathbb{V}%
\]
and if $\mathbb{V}=\mathbb{R}^{2n_{A}}\oplus0$ this is possible for
$S\in\operatorname*{Sp}(n)$ if and only if $S=S_{A}\oplus S_{B}$.

To show that Theorem \ref{Prop1} implies Abbondandolo and Matveyev's result
for arbitrary complex subspaces $\mathbb{V}$ is actually straightforward. To
see this, let $\mathbb{V}^{\sigma}$ be the symplectic orthocomplement of
$\mathbb{V}$ in $(\mathbb{R}^{2n},\sigma).$ Choose symplectic bases
$\mathcal{B}_{\mathbb{V}}$ of $\mathbb{V}$ and $\mathcal{B}_{\mathbb{V}%
^{\sigma}}$ of $\mathbb{V}^{\sigma}$ such that their union $\mathcal{B}$ is a
symplectic orthonormal basis of $\mathbb{R}^{2n}$ (this is easily done using
the symplectic version of the Gram--Schmidt construction for orthogonal bases;
see \cite{Birk}). Set $\dim\mathbb{V}=2n_{A}$ and $\dim\mathbb{V}^{\sigma
}=2n_{B}$ (hence $n=n_{A}+n_{B}$) and let $\mathcal{B}_{A}$ and $\mathcal{B}%
_{B}$ be symplectic bases of $(\mathbb{R}^{2n_{A}},\sigma_{A})$ and
$(\mathbb{R}^{2n_{B}},\sigma_{B})$, respectively, such that $\mathcal{B}%
_{A}\cup\mathcal{B}_{B}$ is a symplectic orthonormal basis of $(\mathbb{R}%
^{2n},\sigma)$. Let $U$ be the linear mapping $\mathbb{R}^{2n}\longrightarrow
\mathbb{R}^{2n}$ defined by $\mathcal{B}_{\mathbb{V}}=U(\mathcal{B}_{A})$ and
$\mathcal{B}_{\mathbb{V}^{\sigma}}=U(\mathcal{B}_{B})$; clearly $U$ is
symplectic and orthogonal. The mapping $\Pi_{\mathbb{V}}:\mathbb{R}%
^{2n}\longrightarrow\mathbb{R}^{2n}$ defined by $\Pi_{\mathbb{V}}=U\Pi
_{A}U^{-1}$ is then the projection onto $\mathbb{V}$ along $\mathbb{V}%
^{\sigma}$ and we thus have
\[
\Pi_{\mathbb{V}}(S(B^{2n}(R)))=U\Pi_{A}U^{-1}(S(B^{2n}(R))).
\]
Theorem \ref{Prop1} implies that there exists $S_{A}^{\prime}\in
\operatorname*{Sp}(n_{A})$ such that%
\[
\Pi_{A}(U^{-1}(S(B^{2n}(R))))\supseteq S_{A}^{\prime}(B^{2n_{A}}(R))
\]
and hence, since $U$ and $S_{A}^{\prime}$ are volume preserving,
\begin{equation}
\operatorname*{Vol}\Pi_{\mathbb{V}}(S(B^{2n}(R)))\geq\frac{(\pi R^{2})^{n_{A}%
}}{n_{A}!}\label{abbonexact1}%
\end{equation}
which is (\ref{abbo}). The \textquotedblleft exact\textquotedblright\ case in
(\ref{abbonexact1}) is equivalent to
\begin{equation}
\operatorname*{Vol}\Pi_{A}(S^{\prime}(B^{2n}(R)))=\frac{(\pi R^{2})^{n_{A}}%
}{n_{A}!}\label{abbonexact}%
\end{equation}
where $S^{\prime}=U^{-1}S$ with $U\in U(n)$. Theorem \ref{Prop1} implies that
the projection $\Pi_{A}(S^{\prime}(B^{2n}(R)))$ is either identical to
$S_{A}^{\prime}(B^{2n_{A}}(R))$ (for some $S_{A}^{\prime}\in\operatorname*{Sp}%
(n_{A})$) or strictly larger than $S_{A}^{\prime}(B^{2n_{A}}(R))$ (in which
case its volume cannot be ${(\pi R^{2})^{n_{A}}}/{n_{A}!}$); hence
(\ref{abbonexact}) holds if and only if
\[
\Pi_{A}(S^{\prime}(B^{2n_{A}}(R)))=S_{A}^{\prime}(B^{2n_{A}}(R))
\]
in which case $S^{\prime}=S_{A}^{\prime}\oplus S_{B}^{\prime}$ for some
$S_{B}^{\prime}\in\operatorname*{Sp}(n_{B})$ (\textit{cf.} Theorem
\ref{Prop1}). Finally, in view of the discussion above we have the
equivalences
\[
S^{\prime}=S_{A}^{\prime}\oplus S_{B}^{\prime}\Leftrightarrow JS^{\prime
T}(\mathbb{R}^{2n_{A}}\oplus0)=S^{\prime T}(\mathbb{R}^{2n_{A}}\oplus
0)\Leftrightarrow JS^{T}\mathbb{V}=S^{T}\mathbb{V}%
\]
where we used the identities $S^{\prime T}=S^{T}U$ and $U(\mathbb{R}^{2n_{A}%
}\oplus0)=\mathbb{V}$.


\section{Perspectives\label{SecPerspectives}}

A first natural question that arises is whether Theorem \ref{Prop1} can be
extended in some way to non-linear symplectic mappings, that is to general
symplectomorphisms of $(\mathbb{R}^{2n},\sigma)$. The first answer is that
there are formidable roadblocks to the passage from the linear to the
nonlinear case, as shortly mentioned in the Introduction. For instance,
Abbondandolo and Matveyev \cite{abbo} show, elaborating on ideas of Guth
\cite{Guth}, that for every $\varepsilon>0$ one can find a symplectomorphism
$f$ of $(\mathbb{R}^{2n},\sigma)$ defined near $B^{2n}(0,1)$ such that
\[
\operatorname*{Vol}(\Pi_{\mathbb{V}}f(B^{2n}(0,1))<\varepsilon\text{.}%
\]
They however speculate in \cite{abbo} that their projection result might still
hold true when the linear symplectic automorphism $S\in\operatorname*{Sp}(n)$
is replaced with a symplectomorphism $f$ of $(\mathbb{R}^{2n},\sigma)$ close
to a linear one. It would be interesting to apply our methods to tackle this
difficult problem.

Also, Theorem \ref{Prop1} could be used to shed some light on packing problems
(see the review \cite{Schlenk} by Schlenk) which form a notoriously difficult
area of symplectic topology.

Given the partitioning of $\mathbb{R}^{2n}= \mathbb{R}^{2n_{A}} \oplus
\mathbb{R}^{2n_{B}}$ it seems natural to expect some connection between
orthogonal projections of symplectic balls and the separability/entanglement
problem in quantum mechanics \cite{go18,Lami,ww}. We intend to address this
problem in a future work.

\vspace{1cm}

*******************************************************************

\textbf{Author's addresses:}

\begin{itemize}
\item \textbf{Nuno Costa Dias} and \textbf{Jo\~ao Nuno Prata:
}Escola Superior N\'autica Infante D. Henrique. Av. Eng.
Bonneville Franco, 2770-058 Pa\c{c}o d'Arcos, Portugal and Grupo
de F\'{\i}sica Matem\'{a}tica, Universidade de Lisboa, Av. Prof.
Gama Pinto 2, 1649-003 Lisboa, Portugal

\item \textbf{Maurice A. de Gosson}: University of Vienna, Faculty of Mathematics (NuHAG), Vienna, Austria.

\end{itemize}

\vspace{0.3cm}

\small

{\it E-mail address} (N.C. Dias): ncdias@meo.pt

{\it E-mail address} (M. de Gosson): maurice.de.gosson@univie.ac.at

{\it E-mail address} (J.N. Prata): joao.prata@mail.telepac.pt

\indent 

**************************************************************************

\end{document}